\documentclass[12pt]{article}   
\usepackage{graphicx,psfrag}
\usepackage{color}
\usepackage{array}
\usepackage{colortbl}
\usepackage{tabularx}
\newtheorem{theorem}{Theorem}
\newtheorem{corollary}[theorem]{Corollary}
\newtheorem{conjecture}[theorem]{Conjecture}

\newtheorem{lemma}[theorem]{Lemma}

\newtheorem{example}[theorem]{Example}

\newtheorem{question}[theorem]{Question}

\usepackage{tikz}
\usetikzlibrary{arrows,chains,matrix,positioning,scopes}
\newenvironment {proof} {{\it
Proof.}}{\hspace*{\fill}$\Box$\par\vspace{4mm}}
\usepackage{amsmath}
\usepackage{amsfonts}

\usepackage{geometry}
\usepackage{graphicx}
\usepackage{amssymb}
\usepackage{mathrsfs}
\usepackage[all]{xy}
\usepackage{amsfonts}
\usepackage{float}
\usepackage{enumerate}
\usepackage{amsmath}
\usepackage{cases}
\usepackage{amsfonts}
\usepackage{graphicx,psfrag}
\usepackage{array}
\usepackage{colortbl}
\usepackage{amsmath}
\usepackage{fancyhdr}
\usepackage{epstopdf}
\usepackage{subfigure}
\usepackage{color}
\usepackage{enumerate}
\usepackage{tikz}
\usepackage{amsmath}
\usetikzlibrary{arrows,chains,matrix,positioning,scopes}
\usepackage{graphicx,psfrag}
\newfont{\bb}{msbm10}

\def\:{\! :\!}

\usepackage{tikz}

\begin{document}

\title{Pattern -Avoiding  $(0,1)$-Matrices}

 \author{Richard A. Brualdi\\
 Department of Mathematics\\
 University of Wisconsin\\
 Madison, WI 53706 USA\\
 {\tt brualdi@math.wisc.edu}
 \and
 Lei Cao\\
 Department of Mathematics\\
 Nova Southeastern  University\\
Ft. Lauderdale, FL 33314 USA\\
 {\tt lcao@nova.edu}
 }

\maketitle

 \begin{abstract} We investigate pattern-avoiding $(0,1)$-matrices as generalizations of pattern-avoiding permutations. Our emphasis  is on 123-avoiding and 321-avoiding patterns for which we obtain exact results as to the maximum number of 1's such matrices can have. We also give algorithms which, when carried out in all possible ways,  construct all of the pattern-avoiding matrices of these two types.
 

\end{abstract}

\smallskip
\smallskip

\section{Introduction}
Let $n\ge 2$ be a positive integer, and let ${\mathcal S}_n$ be the set of permutations of $\{1,2,\ldots,n\}$. Let $k$  be a positive integer with $k\le n$, and let $\sigma=(p_1,p_2,\ldots,p_k)$  be a permutation in ${\mathcal S}_k$.  A permutation $\pi=(i_1,i_2,\ldots,i_n)\in {\mathcal S}_n$ {\it contains the pattern $\sigma$}  provided there exists  $1\le t_1<t_2<\cdots<t_k\le n$ such that  $i_{t_r}<i_{t_s}$  if and only if $p_{t_r}<p_{t_s}$ ($1\le r<s\le k$);
otherwise, $\pi$ is called {\it $\sigma$-avoiding}; we also say that $\pi$ {\it avoids} $\sigma$. Patterns in permutations and, in particular, pattern avoidance, have been extensively studied. If $k=2$ and $\sigma=(1,2)$, then $\sigma$-avoiding means no ascents and the anti-identity permutation $(n,\ldots,2,1)$ is the  only  $\sigma$-avoiding permutation in ${\mathcal S}_n$. If $\sigma=(2,1)$, then $\sigma$-avoiding means no descents and only the identity permutation $(1,2,\ldots,n)$ is $\sigma$-avoiding. There is a chapter on pattern avoidance  in \cite{Bona}.

Pattern avoidance in permutations  is already interesting when $k=3$. Using reversal ($i_j$ goes to position ${n+1-j}$) and complementation ($i_j$ is replaced with $n+1-i_j$), there are only two permutations 
that are inequivalent with respect to pattern avoidance: $(1,2,3)$ (which is equivalent to $(3,2,1)$),
and $(3,1,2)$ (which is equivalent to $(2,1,3)$, $(1,3,2)$, and $(2,3,1)$). 
For each of the six permutations $\sigma$ in ${\mathcal S}_3$  the number of $\sigma$-avoiding permutations is the Catalan number
\[C_n=\frac{{{2n}\choose n}}{n+1}.\]
It is sometime useful to think of $1$ as S(mall), $2$ as M(edium), and $3$ as L(arge). Then SML and LSM are the only two patterns of size 3 that need be considered.

Again let $\sigma=(p_1,p_2,\ldots,p_k)$  be a permutation in ${\mathcal S}_k$. 
Since permutations $\pi=(i_1,i_2,\ldots,i_n)\in {\mathcal S}_n$ and $n\times n$ permutation matrices are in a bijective correspondence:
\[(i_1,i_2,\ldots,i_n)\leftrightarrow P=[p_{ij}]\mbox{ where $p_{k,i_k}=1\ (k=1,2,\ldots,n)$ and $p_{ij}=0$, otherwise},\]
 $\sigma$-avoiding permutations can be studied in terms of permutation matrices.
An {\it $\sigma$-avoiding permutation matrix} is a permutation matrix corresponding to a $\sigma$-avoiding permutation. Thus an $n\times n$ permutation matrix $P$ is (1,2,3)-avoiding provided it does not contain as a submatrix the $3\times 3$ identity matrix 
\begin{equation}\label{eq:123}
I_3=\left[\begin{array}{ccc}
1&0&0\\
0&1&0\\
0&0&1\end{array}\right];\end{equation}
 $P$ is (3,1,2)-avoiding provided it does not contain as a  submatrix the $3\times 3$  matrix 
\begin{equation}\label{eq:312}
\Pi_3=\left[\begin{array}{ccc}
0&0&1 \\ 
1&0&0 \\ 
0&1&0 \end{array}\right].\end{equation}

From now on we usually  use the  briefer notation $j_1j_2\cdots j_k$ for a permutation $(j_1,j_2,\ldots,j_k)\in {\mathcal S}_k$.
Pattern avoidance when expressed in terms of permutation matrices  leads to some interesting questions. 
The $n\times n$ matrix $J_n$ of all 1's can be decomposed into both  123-avoiding permutation matrices and 312-avoiding permutation matrices. In fact, $J_n$ can be decomposed into permutation matrices that are both 123-avoiding and 312-avoiding. This can be done by considering the $n\times n$ permutation matrices corresponding to the $n$ permutations
\[(n,n-1,n-2,\ldots,2,1), (n-1,n-2,\ldots,2,1,n),
(n-2,\ldots,2,1,n,n-1),\ldots,(1,n,n-1,\ldots,2)\]
each of which consists of an decreasing sequence followed by   another decreasing sequence (empty in one case).
For example,
\begin{equation}\label{eq:J4}
J_4=\left[\begin{array}{c|c|c|c}
&&&1\\ \hline
&&1&\\ \hline
&1&&\\ \hline
1&&&\end{array}\right]
\left[\begin{array}{c|c|c|c}
&&1&\\ \hline
&1&&\\ \hline
1&&&\\ \hline
&&&1\end{array}\right]
\left[\begin{array}{c|c|c|c}
&1&&\\ \hline
1&&&\\ \hline
&&&1\\ \hline
&&1&\end{array}\right]+
\left[\begin{array}{c|c|c|c}
1&&&\\ \hline
&&&1\\ \hline
&&1&\\ \hline
&1&&\end{array}\right].\end{equation}
Recall that a permutation $i_1i_2\cdots i_n$ is a {\it Grassmannian} provided that it contains at most one descent (so at most one consecutive 21-pattern) and is a {\it reverse-Grassmannian}
provided that it contains at most one ascent (so at most one consecutive $12$-pattern). The above decomposition  of $J_n$, illustrated in (\ref{eq:J4}),  is a decomposition of $J_n$ into reverse-Grassmanians. A
123-avoiding permutation need not be  a Grassmannian nor a reverse-Grassmanian, for instance, $563412$ is  a 123-avoiding permutation but is not a Grassmannian as it has two decreases nor is it a reverse-Grassmannian as it has two increases.

The notion of   pattern-avoiding permutations (so pattern-avoiding permutation matrices)  has been extended to general $(0,1)$-matrices. 
Given a $k\times k$  permutation matrix $Q$, one seeks the largest number of 1's in an $n\times n$ $(0,1)$-matrix $A$ such that $A$ does not contain a $k\times k$ submatrix $Q'$ such that $Q\le Q'$ (entrywise), that is, $Q'$ has 1's wherever $Q$ has 1's but is allowed to have 1's where $Q$ has $0$'s. Put more simply, the matrix $A$ does not  contain the pattern prescribed by $Q$.
The emphasis here seems to have been on (asymptotic) inequalities, in particular,  the F\H uredi-Hajnal conjecture \cite{Furedi}: {\it Let $P$ be any permutation matrix. Then there is a constant $c_P$ such that the number of $1$'s in  an $n\times n$  $(0,1)$-matrix $A$ that avoids $P$ is bounded by $c_Pn$}. 
Marcus and Tardos \cite{Marcus} solved this conjecture by proving: {\it if $P$ is $k\times k$, then the number of $1$'s in $A$ is bounded by $2k^4 {{k^2}\choose k}n$.} (See also pages 159--164 of \cite{Bona}). There have been considerable investigations on avoiding patterns of various types  in $(0,1)$-matrices, again mostly in terms of asymptotic inequalities. We mention, for instance, the papers \cite{Furedi,  Marcus, SP1, SP2,GT} and the thesis \cite{BK}. 
Our motivation comes from pattern avoiding in permutation matrices and  the possibility of obtaining exact results for specific permutation matrices $Q$ with equality characterizations.
We focus on the two inequivalent patterns of length $3$, namely $123$ and $312$.  In the case of $123$ we can more generally handle the case of the pattern $12\cdots k$ for  arbitrary $k$.
 Also, there is no reason to assume that the matrices $A$ are  square.
 
  In the next section 
we consider  $12\cdots k$-avoiding, $m\times n$ $(0,1)$-matrices. We determine the maximum number $c_{m,n}(12\cdots k)$ of 1's such a matrix can contain, and characterize the matrices achieving this maximum. Surprisingly, it turns out that  one can use a `greedy' algorithm that always results in a 
$12\cdots k$-avoiding, $m\times n$ $(0,1)$-matrix with this maximum number $c_{m,n}(12\cdots k)$ of 1's, and when carried out in all possible ways produces all such matrices.  In the following section we obtain similar results for $312$-avoiding matrices and its maximum number $c_{m,n}(312)$ of 1's. We had hoped to settle  the  case of $k1\cdots(k-1)$ for arbitrary $k$,  but its solution has eluded us thus far, although we are confident of the correct answer. In the final section we consider some questions for future investigations.

\section{$12\cdots k$-avoiding $(0,1)$-Matrices}

Let $A$ be an $m\times n$  $(0,1)$-matrix. The matrix $A$ is {\it $12\cdots k$-avoiding} provided it does not contain a $k\times k$ submatrix of the form 
\begin{equation}\label{eq:123matrix}
\left[\begin{array}{c|c|c|c|c}
1&*&*&*&*\\ \hline
*&1&*&*&*\\ \hline
\vdots&* &\ddots&*&*\\ \hline
*&*&\cdots &1&*\\ \hline
*&*&\cdots&*&1\end{array}\right]\end{equation}
where $*$ denotes either a 0 or a 1. In terms  of the description in the previous paragraph, $Q=I_k$, the $k\times k$ identity matrix. 

We start with the very simple case of $k=2$ which contains many of the ideas used for general $k$. Let $A=[a_{ij}]$ be an $m\times n$ $(0,1)$-matrix with $m,n\ge 2$.
Then  $A$ is  {\it $12$-avoiding}  provided that  $A$ does not have  $2\times 2$ submatrix of the form
\[\left[\begin{array}{cc} 1&*\\ *&1\end{array}\right],\]
where again $*$ denotes a 0 or a 1. A $12$-avoiding matrix can be regarded as a generalization of a non-increasing sequence (permutation) to an arbitrary $(0,1)$-matrix.
The matrix $A$ has $(m+n-1)$ (left-to-right) diagonals (sometimes called {\it Toeplitz diagonals} as the entries on them are constant in Toeplitz matrices) starting from any entry in its first row or first column. If $A$ is $12$-avoiding, each of these diagonals can contain at most one 1 and hence  $A$ can contain at most $(m+n-1)$ 1's. We can attain $(m+n-1)$ 1's with the matrix $A$ whose first row and first column contains  all 1's and all other entries are 0. We call such a $m\times n$  $(0,1)$-matrix a {\it trivial $12$-avoiding matrix}.
Such a matrix contains a $(m-1)\times (n-1)$ zero submatrix $O_{m-1,n-1}$ in its lower right corner.

An $m\times n$ $(0,1)$-matrix $A=[a_{ij}]$ has a {\it complete right-to-left zigzag path}, abbreviated to {\it complete R-L zigzag path},
provided (i) it has $(m+n-1)$ $1$'s,  (ii) $a_{1n}=a_{m1}=1$, and (iii) every 1 of $A$  except for $a_{m1}$
has a 1 immediately to its left or immediately below it, but not both. A {\it complete left-to-right zigzag path} ({\it complete L-R
zigzag path}) is defined in a similar way with $a_{11}=a_{mn}=1$.   An $m\times n$ matrix with a complete R-L zigzag path has exactly $(m+n-1)$\ 1's and $(m-1)(n-1)$ $0$'s. The matrix in Example \ref{ex:12avoiding} below has a complete R-L zigzag path. Notice that its $0$'s  are obtained by starting with an $m\times n$ array containing an $(m-1)\times (n-1)$ zero matrix in 
 its lower right corner, partitioning it into a {\it Ferrers array} and its complement (which we call a  {\it reverse-Ferrers array}), and then translating the Ferrers array one unit in the northwest direction (i.e. up along the diagonals).

\begin{example}\label{ex:12avoiding}{\rm
The following is an   $12$-avoiding $8\times 8$ $(0,1)$-matrix with the maximum number 15 of $1$'s:
\[\left[\begin{array}{c|c|c|c|c|c|c|c}
\phantom{0} &&&&&&&\\ \hline
&\bf 0&\bf 0&\bf 0&\bf 0&\bf 0&\bf 0&\bf 0\\ \hline
&\bf 0&\bf 0&\bf 0&\bf 0&\bf 0&\bf 0&\bf 0\\ \hline
&\bf 0&\bf 0&\bf 0&\bf 0&\bf 0&0&0\\ \hline
&\bf 0&\bf 0&\bf 0&\bf 0&\bf 0&0&0\\ \hline
&\bf 0&\bf 0&0&0&0&0&0\\ \hline
&\bf 0&\bf 0&0&0&0&0&0\\ \hline
&0&0&0&0&0&0&0\end{array}\right]\rightarrow
\left[\begin{array}{c|c|c|c|c|c|c|c}
\bf 0&\bf 0&\bf 0&\bf 0&\bf 0&\bf 0&\bf 0&1\\ \hline
\bf 0&\bf 0&\bf 0&\bf 0&\bf 0&\bf 0&\bf 0&1\\ \hline
\bf 0&\bf 0&\bf 0&\bf 0&\bf 0&1&1&1\\ \hline
\bf 0&\bf 0&\bf 0&\bf 0&\bf 0&1&0&0\\ \hline
\bf 0&\bf 0&1&1&1&1&0&0\\ \hline
\bf 0&\bf 0&1&0&0&0&0&0\\ \hline
1&1&1&0&0&0&0&0\\ \hline
1&0&0&0&0&0&0&0\end{array}\right].\]
}\hfill{$\Box$}\end{example}

 We have the following elementary lemma.

\begin{lemma}\label{lem:12avoiding}
Let $A=[a_{ij}]$ be an  $12$-avoiding $m\times n$  $(0,1)$-matrix with the maximum number $(m+n-1)$ of $1$'s. Then $A$ is obtained from the trivial $m\times n$ $12$-avoiding matrix by translating a Ferrers array one unit in the northwest direction creating a complete R-L zigzag path of $1$'s.
\end{lemma}

\begin{proof} The matrix $A$ has $(m+n-1)$\ 1's with $a_{1n}=1$, and either $a_{1,n-1}=1$ or $a_{2n}=1$ but not both. We assume that $a_{1,n-1}=1$ and $a_{2,n}=0$, since the other case can be handled similarly. We must have $a_{in}=0$ for all $i=3,\ldots,n$ for otherwise we have a 12-pattern. A simple induction now completes the proof.
\end{proof}

Note that a similar result holds for a $21$-avoiding matrix with a `Ferrers array' replaced with `reverse Ferrers array' and `R-L zigzag path' replaced with `L-R zigzag path.

For later use we need to consider a generalization of Lemma \ref{lem:12avoiding}. Let $0\le e_1\le e_2\le \cdots\le e_m\le n$. Define an $m\times n$  {\it staircase matrix} $A(e_1,e_2,\ldots,e_m)$ to be the matrix  with $e_i$ left-justified positions in row $i$ removed for $1\le i\le m$. For instance, with $m=5$ and $n=6$, we have
\[A(3,4,5,5,6)=\left[\begin{array}{c|c|c|c|c|c}
\phantom{X}&\phantom{X}&\phantom{X}&\cellcolor[gray]{0.8}\phantom{X}&\cellcolor[gray]{0.8}\phantom{X}&\cellcolor[gray]{0.8}\phantom{X}\\ \hline
&&&&\cellcolor[gray]{0.8}&\cellcolor[gray]{0.8}\\ \hline
&&&&&\cellcolor[gray]{0.8}\\ \hline
&&&&&\cellcolor[gray]{0.8}\\ \hline
&&&&&\\ \hline
\end{array}\right].\]
Thus we remove a partition of positions from the upper right of an $m\times n$ array. If each position of a staircase matrix is either 0 or 1, then we refer to a {\it $(0,1)$-staircase matrix} and we write $A(e_1,e_2,\ldots,e_m)=[a_{ij}] $.

\begin{lemma}\label{lem:12avoiding2}
Let $A(e_1,e_2,\ldots,e_m)=[a_{ij}] $  be a $12$-avoiding $m\times n$ $(0,1)$-staircase matrix. Then the number of $1$'s in $A$ is at most
\[\max\{e_i+m-i:1\le i\le m\}.\]
\end{lemma}
\begin{proof} Such a collection of 1's has to form a collection of zigzag paths. Starting from the  rightmost position of row  $i$ of the staircase matrix,
the maximum number of 1's possible is $e_i+(m-i)$. Hence the lemma follows. \end{proof} 

\begin{example}\label{ex:newex}{\rm 
Consider the staircase array
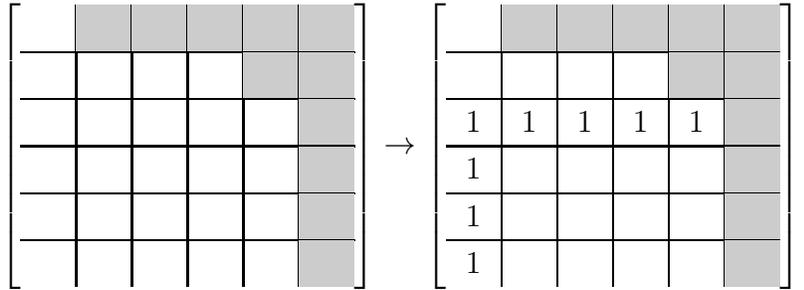
\[
\left[\begin{array}{c|c|c|c|c|c}
\phantom{X}&\cellcolor[gray]{0.8}\phantom{X}&\cellcolor[gray]{0.8}\phantom{X}&\cellcolor[gray]{0.8}\phantom{X}&\cellcolor[gray]{0.8}\phantom{X}&\cellcolor[gray]{0.8}\phantom{X}\\ \hline
&&&&\cellcolor[gray]{0.8}&\cellcolor[gray]{0.8}\\ \hline
&&&&&\cellcolor[gray]{0.8}\\ \hline
&&&&&\cellcolor[gray]{0.8}\\ \hline
&&&&&\cellcolor[gray]{0.8}\\ \hline
&&&&&\cellcolor[gray]{0.8}\end{array}\right]
\rightarrow \left[\begin{array}{c|c|c|c|c|c}
\phantom{X}&\cellcolor[gray]{0.8}\phantom{X}&\cellcolor[gray]{0.8}\phantom{X}&\cellcolor[gray]{0.8}\phantom{X}&\cellcolor[gray]{0.8}\phantom{X}&\cellcolor[gray]{0.8}\phantom{X}\\ \hline
&&&&\cellcolor[gray]{0.8}&\cellcolor[gray]{0.8}\\ \hline
1&1&1&1&1&\cellcolor[gray]{0.8}\\ \hline
1&&&&&\cellcolor[gray]{0.8}\\ \hline
1&&&&&\cellcolor[gray]{0.8}\\ \hline
1&&&&&\cellcolor[gray]{0.8}\end{array}\right].
\]
Then $(e_1,e_2,e_3,e_4,e_f,e_6)=(1,4,5,5,5,5)$ and the maximum in Lemma \ref{lem:12avoiding2} is  8 and it occurs with $i=3$ giving the RL-zigzag pattern above.
}\hfill{$\Box$}
\end{example}

We now consider  the general case of $12\cdots k$-avoiding $m\times n$ $(0,1)$-matrices $A=[a_{ij}]$. First we extend the definition of a complete R-L zigzag path to an {\it R-L zigzag path} where the beginning may be at any entry $a_{pq}$ and the end may be any entry $a_{uv}$ with $u\le p$ and $v\le q$ (this is equivalent to a complete zigzag path in a $(p-u+1)\times (q-v+1)$ submatrix of $A$. In a similar way we have an {\it L-R zigzag path}. 
We illustrate this in the next example.

\begin{example}\label{ex:twozigzags}{\rm 
Let $m=n=6$ and $k=3$. Then the following $(0,1)$-matrices are 123-avoiding, each containing twenty 1's where the 1's are partitioned into a R-L zigzag path of $11$\  $1$'s and a complete R-L zigzag path of $9$\ 1's. 
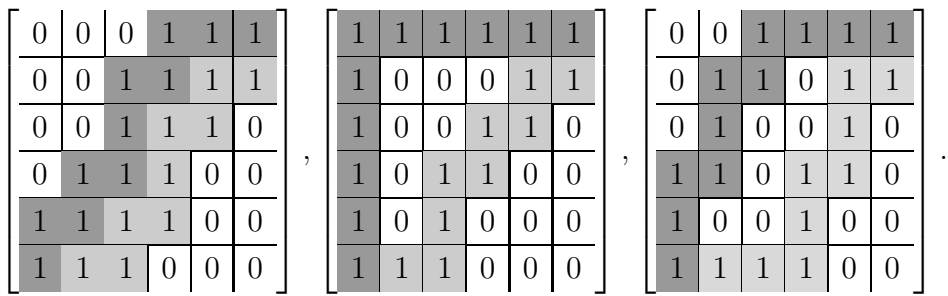
\[\left[\begin{array}{c|c|c|c|c|c}
 0& 0& 0&\cellcolor{gray!80}1&\cellcolor{gray!80}1&\cellcolor{gray!80}1\\ \hline
 0& 0&\cellcolor{gray!80}1&\cellcolor{gray!80}1&\cellcolor{gray!40} 1&\cellcolor{gray!40} 1\\ \hline
 0& 0&\cellcolor{gray!80}1&\cellcolor{gray!40} 1&\cellcolor{gray!40} 1&0\\ \hline
 0&\cellcolor{gray!80}1&\cellcolor{gray!80}1&\cellcolor{gray!40} 1&0&0\\ \hline
\cellcolor{gray!80}1&\cellcolor{gray!80}1&\cellcolor{gray!40} 1&\cellcolor{gray!40} 1&0&0\\ \hline
\cellcolor{gray!80}1&\cellcolor{gray!40} 1&\cellcolor{gray!40} 1&0&0&0\end{array}\right],
\left[\begin{array}{c|c|c|c|c|c}
\cellcolor{gray!80}1&
\cellcolor{gray!80}1&\cellcolor{gray!80}1&
\cellcolor{gray!80}1&\cellcolor{gray!80}1&\cellcolor{gray!80}1\\ \hline
\cellcolor{gray!80}1&0& 0& 0&\cellcolor{gray!40} 1&\cellcolor{gray!40} 1\\ \hline
\cellcolor{gray!80}1& 0& 0&\cellcolor{gray!40} 1&\cellcolor{gray!40} 1&0\\ \hline
\cellcolor{gray!80}1& 0& \cellcolor{gray!40}  1&\cellcolor{gray!40} 1&0&0\\ \hline
\cellcolor{gray!80}1& 0&\cellcolor{gray!40} 1&0&0&0\\ \hline
\cellcolor{gray!80}1&\cellcolor{gray!40} 1&\cellcolor{gray!40} 1&0&0&0\end{array}\right],
\left[\begin{array}{c|c|c|c|c|c}
0&0&\cellcolor{gray!80}1&\cellcolor{gray!80}1&\cellcolor{gray!80}1&\cellcolor{gray!80}1\\ \hline
0&\cellcolor{gray!80}1&\cellcolor{gray!80}1&0&\cellcolor{gray!30}1&\cellcolor{gray!30}1\\ \hline

0&\cellcolor{gray!80}1&0&0&\cellcolor{gray!30}1&0\\ \hline

\cellcolor{gray!80}1&\cellcolor{gray!80}1&0&\cellcolor{gray!30}1&\cellcolor{gray!30}1&0\\ \hline
\cellcolor{gray!80}1&0&0&\cellcolor{gray!30}1&0&0\\ \hline
\cellcolor{gray!80}1&\cellcolor{gray!30}1&\cellcolor{gray!30}1&\cellcolor{gray!30}1&0&0\end{array}\right].\]
}\hfill{$\Box$}\end{example}

 \begin{lemma}\label{lem:kmax}
Let $m,n,k$ be positive integers with $m,n\ge k\ge 3$.  Then the maximum number of $1$'s in a $12\cdots k$-avoiding $m\times n$ $(0,1)$-matrix $A$ is $(k-1)(m+n-(k-1))$. 
 \end{lemma}     		
   
     		\begin{proof}  		
      		 There are  two diagonals of $A$ of each length $1,2,\ldots,(k-2)$ and $(m+n-2k+3)$ diagonals of length at least $(k-1)$. Thus the number of 1's in $A$ is at  most
      		\[2(1+2+\cdots+(k-2)) +(k-1)(m+n-2k+3)=(k-1)(m+n-(k-1)).\]
      		It is straightforward to construct an $m\times n$ $(0,1)$-matrix which is $12\cdots k$-avoiding and contains this many $1$'s. One can simply mimic the construction of the first matrix  in Example \ref{ex:twozigzags}. For instance, if $m=6$, $n=8$ and $k=4$, the following matrix is $1234$-avoiding with $33$\ 1's:
      		\[
      		\left[\begin{array}{c|c|c|c|c|c|c|c}
  &&&&&1&1&1\\ \hline   
   &&&&1&1&1&1\\ \hline   
    &1&1&1&1&1&1&1\\ \hline   
     1&1&1&1&1&1&1&1\\ \hline   
      1&1&1&1&1&1&1&\\ \hline   
       1&1&1&1&&&&\end{array}\right].\] 		
     \end{proof}
     
     We now prove the main theorem of this section.

      \begin{theorem}\label{th:no123} Let $m, n,k$ be integers with $m,n\ge k\ge 3$.
      
      \begin{itemize}
      \item[\rm (i)] The maximum number of $1$'s in an $m\times n$ 
$(0,1)$-matrix which is  $12\cdots k$-avoiding  is \\ $(k-1)(m+n-(k-1))$. 
\item[\rm (ii)] Every
 $12\cdots k$-avoiding $m\times n$ $(0,1)$-matrix with $(k-1)(m+n-(k-1))$\ $1$'s can be decomposed into $(k-1)$ R-L zigzag paths of lengths
 $m+n-1, m+n-3, m+n-5,\ldots, m+n-(2k-3)$, respectively. 
 \item[\rm (iii)] If $A$ is a 
$12\cdots k$-avoiding $m\times n$ 
$(0,1)$-matrix containing fewer than $(k-1)(m+n-(k-1))$\ $1$'s, then there is a $0$ in $A$ that can be replaced with a $1$ so that the resulting matrix remains $12\cdots k$-avoiding.
\end{itemize}
      \end{theorem}
 
   \begin{proof}   The first assertion (i) of the theorem is Lemma \ref{lem:kmax}, and we now consider assertion (ii). 
Let $A$ be a  $12\cdots k$-avoiding $m\times n$ $(0,1)$-matrix with $(k-1)(m+n-(k-1))$\ $1$'s.
By maximality, we must have 1's in the $(1,n-1), (1,n), (2,n)$ positions. 
Moreover, row 1 consists of a string of $0$'s followed by a string  of $1$'s. To see this, suppose that $a_{1,p-1}=1$ and $a_{1p}=0$.  Then the diagonal  determined by $a_{1p}$ either contains fewer  1's then a maximal matrix contains or it contains $(k-1)$ $1$'s,  contradicting the assumption that $A$ is $12\cdots k$-avoiding. Thus there is a $p\ge n-1$ such that $a_{11}=\cdots= a_{1,p-1}=0$ and $a_{1p}=\cdots=a_{1n}=1$. If $a_{2p}=0$ then, because $a_{1p}=1$,  maximality again is contradicted since we could change $a_{2p}$ to 1; hence $a_{2p}=1$. We now repeat the above argument with $a_{2p}$ replacing $a_{1p}$, and inductively construct a a complete R-L zigzag path connecting $a_{1n}$ and $a_{m1}$. This zigzag path contains a 1 from each diagonal of $A$. Let $A'$ be the matrix obtained from $A$ by replacing each 1 on this zigzag path with a $0$ and deleting row 1 and column 1.  Then $A'$ is a  $12\cdots(k-1)$-avoiding  $(m-1)\times (n-1)$ $(0,1)$-matrix and by induction can  be decomposed into $(k-2)$ zigzag paths of lengths $m+n-3,\cdots,m+n-(2k-3)$ zigzag paths. Thus assertion (ii) holds.

We now  prove assertion (iii) of the theorem. Let $A=[a_{ij}]$  be a $12\cdots k$-avoiding $m\times n$ matrix with fewer than  $(k-1)(m+n-(k-1))$\ $1$'s. Thus one of the diagonals of $A$ does not contain the maximum allowable 1's in a $12\cdots k$-avoiding $(0,1)$-matrix. Consider the leftmost 1's in each row of $A$. If these leftmost 1's form a  complete R-L zigzag path connecting $a_{1n}$ with $a_{m1}$ we replace these 1's with 0's and delete row 1 and column 1, resulting in 
an $12\ldots (k-1)$-avoiding $(m-1)\times (n-1)$ $(0,1)$-matrix  with the maximum number of 1's, and then use induction.  Now suppose that these leftmost 1's do not form such a zigzag path. Starting from $a_{1n}=1$ we follow a zigzag path, alternating between going as far left as possible then as far down  as possible; so we reach a $a_{kl}=1$ such that both $a_{k,l-1}=0$ and  $a_{k+1,l}=0$.
Suppose that $a_{kp}=1$ for some $p$ with $1\le p\le l-2$. Then there cannot exist a $12\cdots (k-1)$-pattern below and to the right of $a_{k,l-1}$ since we would then have a $12\cdots k$-pattern. Hence  $a_{k,p+1}=\cdots=a_{kl}=1$ and we can continue our zigzag path. Now replacing $a_{k+1,l}=0$ with a 1 cannot create a $12\cdots k$-pattern for if it did, $a_{kl}=1$ would have been part of a $12\cdots k$-pattern. Thus we can increase the number of 1's in $A$ if $A$ has fewer  than $(k-1)(m+n-(k-1))$\ $1$'s, and this proves (iii).
\end{proof}
       
      An $m\times n$ $(0,1)$-matrix $A$ is  {\it maximal $12\cdots k$-avoiding} provided $A$ is $12\cdots k$-avoiding and every $(0,1)$-matrix obtained from $A$ by replacing a 0 with a 1 is not $12\cdots k$-avoiding, that is, by Theorem \ref{th:no123}, $A$ has fewer than $(k-1)(m+n-(k-1)$\ $1$'s.  By (iii) of Theorem \ref{th:no123}, all maximal matrices contain the same number of 1's; in the context of graph theory, this can be viewed as a `saturation' result. It follows from Theorem \ref{th:no123}  that the following `greedy' algorithm
      can always be carried out  to produce a $12\cdots k$-avoiding matrix with the maximum number $(k-1)(m+n-(k-1)$ of  1's:
      
      \bigskip
      \centerline{\bf Algorithm for a Maximal $12\cdots k$-avoiding Matrix}
      \begin{itemize}
      \item[\rm (i)] Let $A$ be the $m\times n$ zero matrix.
      \item[\rm (ii)] While $A$ has less than $(k-1)(m+n-(k-1)$\ 1's, locate any 0 of $A$ whose replacement with a 1 results in a $12\cdots k$-avoiding matrix.
      \item[\rm (iii)] Output $A$.
      \end{itemize}
      
      To complete this section we make the following observation.
      It follows from Theorem \ref{th:no123} that a maximal $12\cdots k$-avoiding $m\times n$ $(0,1)$-matrix can be decomposed into $(k-1)$ R-L zigzag paths of lengths
 $m+n-1, m+n-3, m+n-5,\ldots, m+n-(2k-3)$. However, one cannot arbitrarily choose such R-L zigzag paths of these lengths and be assured of obtaining  a maximal $12\cdots k$-avoiding $(0,1)$-matrix.  This is illustrated in the next example.
 
  \begin{example}{\rm \label{ex:notarbzigzag}	Let $m=n=8$ and $k=5$, and consider   12345-avoiding $8\times 8$ $(0,1)$-matrices. A maximal such matrix can be decomposed into R-L zigzag paths of lengths $15, 13, 11,\mbox{ and }9$,  respectively.
 Suppose we attempt to construct such a matrix by  choosing R-L zigzag paths of lengths 
      	$15$, $13$, and $11$ and  obtain a 12345-avoiding $(0,1)$-matrix. Below we indicate a choice of such paths by $a$, $b$, and $c$, respectively:      
  \[\left[\begin{array}{c|c|c|c|c|c|c|c}
  &&&&b&b&b&a\\ \hline    
  &&&&b&&a&a\\ \hline    
  &&b&b&b&a&a&c\\ \hline    
  &&b&&a&a&&c\\ \hline    
  b&b&b&a&a&&c&c\\ \hline    
  b&&a&a&&c&c&\\ \hline    
  b&a&a&&c&c&&\\ \hline    
  a&a&c&c&c&&&\end{array}\right].\]
As is easily checked,    there is not any R-L zigzag path of length 9 that can now be inserted. But we can put in 9\ 1's, indicate by $d$ below, so that the result is still 12345-avoiding, and then we can decompose the resulting matrix  into zigzag paths of lengths $15,13,11,\mbox{ and }9$ indicated by  $x,y,z,\mbox{ and }u$:
  \[\left[\begin{array}{c|c|c|c|c|c|c|c}
  d&d&d&d&b&b&b&a\\ \hline    
  d&&&&b&d&a&a\\ \hline    
  d&&b&b&b&a&a&c\\ \hline    
  d&&b&&a&a&&c\\ \hline    
  b&b&b&a&a&&c&c\\ \hline    
  b&d&a&a&&c&c&\\ \hline    
  b&a&a&&c&c&&\\ \hline    
  a&a&c&c&c&&&\end{array}\right]\rightarrow
  \left[\begin{array}{c|c|c|c|c|c|c|c}
  x&x&x&x&x&x&x&x\\ \hline
  x&&&&y&y&y&y\\ \hline
  x&&y&y&y&z&z&z\\ \hline
  x&&y&&z&z&&u\\ \hline
  x&y&y&z&z&&u&u\\ \hline
  x&y&z&z&&u&u&\\ \hline
  x&y&z&&u&u&&\\ \hline
  x&y&z&u&u&&&\end{array}\right].\]
  
  Moreover, if one has a maximal $12\cdots k$-avoiding $m\times n$ $(0,1)$-matrix which,  by Theorem \ref{th:no123}, can be decomposed into R-L zigzag paths of lengths $m+n-1, m+n-3, m+n-5,\ldots, m+n-(2k-3)$,), one cannot choose these paths arbitrarily. For instance,
  let $m=n=6$ and $k=3$. and  consider the maximal 123-avoiding matrix with $11+9=20$\ $1$'s below. If we choose the R-L zigzag path of length 11 indicated by the shaded 1's, then there is not a R-L zigzag path of length 9 that can be formed from the remaining 1's.
  \[\left[\begin{array}{c|c|c|c|c|c}
  &&&1&1&1\\ \hline
  &&1&1&1&1\\ \hline
  &1&1&1&1&\\ \hline
  1&1&1&1&&\\ \hline
  1&1&1&&&\\ \hline
 1&1&&&&\end{array}\right]\rightarrow
 \left[\begin{array}{c|c|c|c|c|c}
  &&&1&\cellcolor{gray!20} 1&\cellcolor{gray!20} 1\\ \hline
  &&1&\cellcolor{gray!20} 1&\cellcolor{gray!20} 1&1\\ \hline
  &1&\cellcolor{gray!20} 1&\cellcolor{gray!20}1&1&\\ \hline
  1&\cellcolor{gray!20}1&\cellcolor{gray!20} 1&1&&\\ \hline
  \cellcolor{gray!20} 1&\cellcolor{gray!20} 1&1&&&\\ \hline
 \cellcolor{gray!20} 1&1&&&&\end{array}\right].\]
  }\hfill{$\Box$}
\end{example}

\section{$312$-avoiding $(0,1)$-Matrices}

An $m\times n$ $(0,1)$-matrix $A$ is {\it maximal $312$-avoiding} provided it is $312$-avoiding  and every matrix obtained from $A$ by replacing a 0 with a 1 is not $312$-avoiding. 
As for the case of $12\cdots k$-avoiding we shall see that all maximal $312$-avoiding $m\times n$ $(0,1)$-matrices have the same number of 1's (and indeed the same number as maximal $123$-avoiding matrices). We also give an algorithm which, carried out in all possible ways, produces all maximal $312$-avoiding $(0,1)$-matrices.

 
  There are three (and only three) $m\times n$ $(0,1)$-matrices which are 312-avoiding and which have an $(m-2)\times (n-2)$ zero submatrix. This zero matrix can be in the upper right corner (the `3' position), the middle left side (the `1' position), and the lower middle (the `2' position). We illustrate these for $m=n=5$ in the next example.
  
   \begin{example}{\rm  \label{ex:three}  Three $312$-avoiding $5\times 5$ $(0,1)$ matrices with a $3\times 3$ zero submatrix:
  \[
  \left[\begin{array}{c|c|c|c|c}
  1& 1&0&0&0\\ \hline
  1 &1&0&0&0\\ \hline
    1&1&0&0&0\\ \hline
     1&1&1&1&1\\ \hline
      1&1&1&1&1\end{array}\right], \ 
      \left[\begin{array}{c|c|c|c|c}
 1 &1&1&1&1\\ \hline
   0&0&0&1&1\\ \hline
    0&0&0&1&1\\ \hline
     0&0&0&1&1\\ \hline
      1&1&1&1&1\end{array}\right], \ 
      \left[\begin{array}{c|c|c|c|c}
  1&1&1&1& 1\\ \hline
   1&1&1&1&1\\ \hline
   1 &0&0&0&1\\ \hline
     1&0&0&0&1\\ \hline
     1 &0&0&0&1\end{array}\right].\]
       The matrices
\[
\left[\begin{array}{c|c|c|c|c}
1&1&1&1&0\\ \hline
0&1&1&1&0\\ \hline
1&1&0&1&0\\ \hline
1&0&0&1&1\\ \hline
1&0&0&1&1\end{array}\right]\mbox{ and }
\left[\begin{array}{c|c|c|c|c|c}
1&1&1&1&0&0\\ \hline
0&1&1&1&0&0\\ \hline
1&1&0&1&0&0\\ \hline
1&0&0&1&1&1\\ \hline
1&0&0&1&1&1\\ \hline
1&0&0&0&0&1\end{array}\right]\] 
are both $312$-avoiding.
     All five of these matrices are easily checked to be maximal $312$-avoiding matrices.
  }\hfill{$\Box$}\end{example}
  
 The position in the upper right corner of an $m\times n$  $(0,1)$-matrix $A$ is special for the property of being 312-avoiding. In fact, it is easy to see that if there is a 1 in the upper right corner, then  the matrix $A$ is  312-avoiding if and only if there does not exist  a 312-pattern using that 1, equivalently, if and only if the $(m-1)\times (n-1)$- submatrix $A(1,n)$ obtained from $A$ by deleting row 1 and
 column n is 12-avoiding. If, in addition, $A$ is a maximal 312-avoiding matrix, then the first row and last column of $A$ contain only 1's. Thus by Lemma \ref{lem:12avoiding}, $A(1,n)$ 
  contains exactly $((m-1)+(n-1)-1)=(m+n-3)$\ 1's and $A$ contains exactly $(m+n-1)+(m+n-3)=2(m+n-2)$\ 1's; in addition,  the 1's in $A(1,n)$ form a complete L-R zigzag path.

\begin{example}{\rm \label{ex:upper}
Let $m=8$ and $n=10$. Then an $8\times 10$\   312-avoiding 
$(0,1)$-matrix with $32$\ $1$'s as discussed above is  
\[\left[\begin{array}{c|c|c|c|c|c|c|c|c|c}
1&1&1&1&1&1&1&1&1&\mathbf 1\\ \hline
0&0&0&0&0&0&0&\cellcolor{gray!20}1&\cellcolor{gray!20} 1&1\\ \hline
0&0&0&0&0&0&\cellcolor{gray!20}1&\cellcolor{gray!20} 1&0&1\\ \hline
0&0&0&0&\cellcolor{gray!20} 1&\cellcolor{gray!20} 1&\cellcolor{gray!20} 1&0&0&1\\ \hline
0&0&0&\cellcolor{gray!20} 1&\cellcolor{gray!20} 1&0&0&0&0&1\\ \hline
0&0&\cellcolor{gray!20} 1&\cellcolor{gray!20} 1&0&0&0&0&0&1\\ \hline
0&0&\cellcolor{gray!20} 1&0&0&0&0&0&0&1\\ \hline
\cellcolor{gray!20} 1&\cellcolor{gray!20} 1&\cellcolor{gray!20} 1&0&0&0&0&0&0&1\end{array}\right],\] 
where the L-R zigzag path is given by the shaded 1's.

Three examples of maximal 312-avoiding matrices where the upper right is the same reverse Ferrers array of 0's   are given below; the differences are highlighted by shading for contrast, where the  more  darkly shaded  1's are those 1's that border the reverse Ferrers array. As we show in the lemma that follows, in a maximal 312-avoiding $(0,1)$-matrix there is always a reverse Ferrers array of 0's in the upper right corner bordered by a R-L zigzag path of 1's. Notice that there is a 1 in the lower left position, and this always holds in a maximal 312-avoiding matrix.}

\[\left[\begin{array}{c|c|c|c|c|c|c|c}
\cellcolor{gray!80}1&\cellcolor{gray!80} 1&0&0&0&0&0&0\\ \hline
1&\cellcolor{gray!80} 1&\cellcolor{gray!80} 1 &0&0&0&0&0\\ \hline
\cellcolor{gray!20}1&1&\cellcolor{gray!80} 1&\cellcolor{gray!80} 1 &0&0&0&0\\ \hline
\cellcolor{gray!20}1&&1&\cellcolor{gray!80} 1&\cellcolor{gray!80} 1 &0&0&0\\ \hline
\cellcolor{gray!20}1&&&1&\cellcolor{gray!80} 1&\cellcolor{gray!80} 1&0&0\\ \hline
\cellcolor{gray!20}1&&&&1&\cellcolor{gray!80} 1&\cellcolor{gray!80} 1&0\\ \hline
\cellcolor{gray!20}1&&&&&1&\cellcolor{gray!80} 1&\cellcolor{gray!80}1\\ \hline
1&&&&&&1&\cellcolor{gray!80} 1\end{array}\right],
\left[\begin{array}{c|c|c|c|c|c|c|c}
\cellcolor{gray!80} 1&\cellcolor{gray!80} 1&0&0&0&0&0&0\\ \hline
1&\cellcolor{gray!80} 1&\cellcolor{gray!80} 1 &0&0&0&0&0\\ \hline
&1&\cellcolor{gray!80} 1&\cellcolor{gray!80} 1 &0&0&0&0\\ \hline
&&1&\cellcolor{gray!80} 1&\cellcolor{gray!80} 1 &0&0&0\\ \hline
&&\cellcolor{gray!20}1 &1&\cellcolor{gray!80} 1&\cellcolor{gray!80} 1&0&0\\ \hline
\cellcolor{gray!20}1&\cellcolor{gray!20}1&\cellcolor{gray!20}1&&1&\cellcolor{gray!80} 1&\cellcolor{gray!80} 1&0\\ \hline
\cellcolor{gray!20}1&&&&&1&\cellcolor{gray!80} 1&\cellcolor{gray!80} 1\\ \hline
1&&&&&&1&\cellcolor{gray!80} 1\end{array}\right],
\left[\begin{array}{c|c|c|c|c|c|c|c}
\cellcolor{gray!80} 1&\cellcolor{gray!80} 1&0&0&0&0&0&0\\ \hline
1&\cellcolor{gray!80} 1&\cellcolor{gray!80} 1 &0&0&0&0&0\\ \hline
\cellcolor{gray!20}1&1&\cellcolor{gray!80} 1&\cellcolor{gray!80} 1 &0&0&0&0\\ \hline
&&1&\cellcolor{gray!80} 1&\cellcolor{gray!80} 1 &0&0&0\\ \hline
\cellcolor{gray!20}1&&\cellcolor{gray!20} 1&1&\cellcolor{gray!80} 1&\cellcolor{gray!80} 1&0&0\\ \hline
&&&&1&\cellcolor{gray!80} 1&\cellcolor{gray!80} 1&0\\ \hline
\cellcolor{gray!20}1&&&&\cellcolor{gray!20}1&1&\cellcolor{gray!80} 1&\cellcolor{gray!80} 1\\ \hline
1&&&&&&1&\cellcolor{gray!80} 1\end{array}\right].\]
\hfill{$\Box$}\end{example}

The next lemma works more generally for  $k12\cdots (k-1)$-avoiding matrices.

\begin{lemma}\label{lem:new2}
Let  $k\ge 3$ and let $A=[a_{ij}]$ be an $m\times n$ $(0,1)$-matrix which is maximal $k12\cdots(k-1)$-avoiding such that $A$ has a $0$ in its upper right corner $(a_{1n}=0)$. Then the  $0$'s in the upper-right hand corner form   a reverse Ferrers array bordered by a complete L-R zigzag path of $(m+n-1)$\ $1$'s.
\end{lemma}

\begin{proof} The maximality property implies that $a_{11}=a_{12}=\cdots=a_{1,k-1}=1$ and that $a_{m-k+2,n}=a_{m-k+3,n}=\cdots=a_{m,n}=1$, since their positions can never occur in a 
$k1\cdots(k-1)$-pattern. By maximality there cannot be a 0 followed by a 1 in row 1; since if replacing that 0 by a 1 creates a $k12\cdots(k-1)$-pattern, then that 1 itself would be part of a $k12\cdots(k-1)$-pattern. We will use both the maximality property  and the $k12\cdots(k-1)$-avoiding property without specifically referring to them now, as it should be clear when they are being invoked. Let $p_1$ be defined by $a_{11}=\cdots=a_{1,p_1}=1, a_{1,p_1+1}=a_{1,p_1+2}=\cdots=a_{1n}=0$.  
We then must have $a_{2,p_1}=1$.   There exists  $p_2> p_1$ such that 
$a_{2,p_1}=a_{2,p_1+1}=\cdots=a_{2p_2}=1$ and $a_{2,p_2+1}=\cdots=a_{2n}=0$. Starting from $p_2$, in a similar way we find $p_3\ge p_2$ such that
$a_{3,p_2}=a_{3,p_2+1}=\cdots=a_{3,p_3}=1$ and $a_{3,p_3+1}=\cdots=a_{3n}=0$. Continuing like this we obtain a reverse Ferrers array of 0's in the upper right corner and a complete R-L zigzag path of $(m+n-1)$\ 1's bordering it which begins with $a_{11}=1$ and  $a_{12}=1$, and ends at $a_{m-1,n}=1$ and $a_{mn}=1$. 
\end{proof}

Referring to the complete R-L zigzag path in Lemma \ref{lem:new2}, we define the 1's at the intersection of the horizontal segments of the path with its vertical segments to be the {\it crucial} 1's of the path. Note that in case of a 1 in the upper right corner, we also have a complete R-L zigzag path (with an empty reverse Ferrers diagram of 0's). We usually boldface the crucial 1's. The matrix $A$ contains a 312-pattern if and only if there is a 312-pattern using a crucial 1 (in the 3-position). Each crucial 1 has an associated {\it corner position}; these are the positions directly southwest of the crucial 1. In a maximal 312-avoiding matrix, the corner positions are all occupied by 1's called {\it corner} 1's; the reason is that if there were a 312-pattern using a corner 1, then the corner 1 can be replaced by its corresponding crucial 1 and give another 312-pattern.

\begin{example}{\rm \label{ex:crucial} We illustrate a complete R-L zigzag path with its crucial and corner 1's. The zigzag path is dark-shaded with its crucial 1's in boldface; the corner 1's are those in the light-shaded positions.
\[\left[\begin{array}{c|c|c|c|c|c|c|c|c|c|c|c}
\cellcolor[gray]{0.7}1&\cellcolor[gray]{0.7}1&\cellcolor[gray]{0.7}1&\cellcolor[gray]{0.7}\mathbf 1&&&&&&&&\\ \hline
&&\cellcolor[gray]{0.9} 1&\cellcolor[gray]{0.7}1&\cellcolor[gray]{0.7}1&\cellcolor[gray]{0.7}\mathbf 1&&&&&&\\ \hline
&&&&\cellcolor[gray]{0.9} 1&\cellcolor[gray]{0.7}1&&&&&&\\ \hline
&&&&&\cellcolor[gray]{0.7}1&\cellcolor[gray]{0.7}1&\cellcolor[gray]{0.7}\mathbf 1&&&\\ \hline
&&&&&&\cellcolor[gray]{0.9} 1&\cellcolor[gray]{0.7}1&\cellcolor[gray]{0.7}\mathbf 1&&&\\ \hline
&&&&&&&\cellcolor[gray]{0.9}1&\cellcolor[gray]{0.7}1&&&\\ \hline
&&&&&&&&\cellcolor[gray]{0.7}1&\cellcolor[gray]{0.7}\mathbf 1 &&\\ \hline
&&&&&&&&\cellcolor[gray]{0.9}1&\cellcolor[gray]{0.7}1&\cellcolor[gray]{0.7}1&\cellcolor[gray]{0.7}\mathbf 1\\ \hline
&&&&&&&&&&\cellcolor[gray]{0.9}1&\cellcolor[gray]{0.7}1\end{array}\right].\]
Notice that the corner 1's give induce another complete R-L zigzag path excluding the crucial 1's.
}\hfill{$\Box$}\end{example}

\begin{theorem}\label{th:312theorem}
Let $A=[a_{ij}]$ be a  $312$-avoiding  $m\times n$\  $(0,1)$-matrix. 
\begin{itemize}
\item[\rm (i)] 
$A$ has at most $2(m+n-2)$ $1$'s. 
\item[\rm (ii)] The maximum number of $1$'s in a  $312$-avoiding  $m\times n$\  $(0,1)$-matrix equals  $2(m+n-2)$.
\item[\rm (iii)] An $m\times n$  maximal $312$-avoiding $(0,1)$-matrix contains exactly  $2(m+n-2)$ $1$'s.
\end{itemize}
\end{theorem}

\begin{proof}  Assume $A$ is maximal 312-avoiding. By Lemma \ref{lem:new2},
$A$ contains a complete R-L zigzag path of $(m+n-1)$\ $1$'s connecting the $(1,1)$ then $(1,2)$-positions with the $(m-1,n)$ then $(m,n)$-positions, where there is a 1 in all positions of this path, and above and to the right of  this path are all $0$'s forming a reverse Ferrers array. All corner positions are occupied by 1's. Now delete row $1$ and column $n$ of $A$ to obtain an  $(m-1)\times (n-1)$ matrix  $A'$ whose unspecified entries are the corresponding  unspecified entries of  $A$ (that is, the unspecified entries of $A$ below the zigzag path).  The matrix $A'$ has nonvacuous rows and columns. Since $A$ does not contain a $312$-pattern, 
 below the zigzag path there cannot be  a $2\times 2$ submatrix of the form
\begin{equation}\label{eq:bipartite}\left[\begin{array}{cc} 1&1\\ *&1\end{array}\right]\end{equation}
 for then the two diagonal 1's of this submatrix, along with a crucial 1,  would be part of a 312-pattern.
 
The 1's of $A$ below the zigzag path (and so the corresponding 1's of $A'$) determine a biadjacency matrix of a bipartite graph $G$ contained in the complete bipartite graph $ K_{m-1,n-1}$ whose edges correspond to these 1's. 
We claim that  $G$ cannot contain a cycle. For, if $G$ contained a cycle $\gamma$ (corresponding therefore to a collection of 1's with exactly two 1's in a certain  set of rows and columns), we see that $A$, and hence $A'$, contains a submatrix of the form (\ref{eq:bipartite}) below the zigzag path: Starting from the  lowest 1 in the rightmost column with a 1, there is a 1 above it, then a 1 to the left of this second 1, and this 
determines  a submatrix of the form 
(\ref{eq:bipartite}). Thus $G$ cannot have any cycles and so, since $G$ has $(m+n-2)$ vertices,  there are at  most $(m+n-3)$\ 1's of $A$ below the zigzag path. Hence the number of 1's of $A$ is at most
\[(m+n-1)+(m+n-3)=2(m+n-2).\]
This proves  assertion  (i) of the theorem.

To prove assertion (ii) we show how to construct a canonical 312-avoiding  $m\times n$\  $(0,1)$-matrix with  $2(m+n-2)$ $1$'s. 
Once we have chosen the 1's of the complete R-L zigzag path (and so the 0's above it), we are left with a staircase array.
We can insert  1's  in this staircase array and  obtain a 312-avoiding $(0,1)$-matrix as long  as the new 1's do not create a pattern as in (\ref{eq:bipartite}) and, in particular, correspond to a tree.

To obtain a maximal 312-avoiding matrix, the additional 1's (including the corner 1's) should determine a spanning tree of $K_{m-1,n-1}$ without a submatrix of the form (\ref{eq:bipartite}) 
giving therefore $(m+n-3)$ more 1's for a total of $(m+n-1)+(m+n-3)=2(m+n-2)$\ 1's. 
One way of doing this is to form the {\it vertical 
shadow}
 of any R-L zigzag path illustrated below and then augment with the remaining elements in row $n$  with the result being the 1's in the mediums-shaded positions:
 \[\left[\begin{array}{c|c|c|c|c|c|c|c|c|c|c|c}
1&1&1&1&\cellcolor[gray]{0.8} \mathbf1&&&&&&&\\ \hline
  &&&\cellcolor[gray]{0.6} 1&1&&&&&&&\\ \hline
   &&&\cellcolor[gray]{0.6} 1&1&&&&&&&\\ \hline
    &&&\cellcolor[gray]{0.6} 1&1&1&1&\cellcolor[gray]{0.8} \mathbf 1&&&&\\ \hline
     &&&&&&\cellcolor[gray]{0.6} 1&1&&&&\\ \hline
      &&&&&&\cellcolor[gray]{0.6} 1&1&1&\cellcolor[gray]{0.8} \mathbf 1&&\\ \hline
       &&&&&&&&\cellcolor[gray]{0.6} 1&1&&\\ \hline
        &&&&&&&&\cellcolor[gray]{0.6} 1&1&&\\ \hline
         &&&&&&&&\cellcolor[gray]{0.6} 1&1&1&\cellcolor[gray]{0.8} \mathbf 1\\ \hline
          &&&&&&&&&&\cellcolor[gray]{0.6} 1&1\\ \hline
           &&&&&&&&&&\cellcolor[gray]{0.6} 1&1\\ \hline
           \cellcolor[gray]{0.6}1&\cellcolor[gray]{0.6}1&\cellcolor[gray]{0.6}1&\cellcolor[gray]{0.6}1&\cellcolor[gray]{0.6}1&\cellcolor[gray]{0.6}1&\cellcolor[gray]{0.6}1&\cellcolor[gray]{0.6}1&\cellcolor[gray]{0.6}1&\cellcolor[gray]{0.6}1&\cellcolor[gray]{0.6} 1&1\end{array}\right].
            \]
This gives an additional $(m-1)+(n-2)=m+n-3$\ 1's for  a total number of 1's equal to
\[(m+n-1)+(m-1)+(n-2)=2(m+n-2),\]
as wanted, and it is 312-avoiding since the crucial 1's are not part of a 312-pattern.            
This shows that the  upper bound in (i) can be attained and completes the proof of assertion (ii).

We now prove assertion (iii). Assume that $A$ is maximal with a given complete R-L zigzag path of $(m+n-1)$\ 1's connecting the $(1,1)$ position to the $(m,n)$ position. Since $A$ is maximal, its $(m,1)$-entry equals 1. If all other entries in row $m$ equal 0, then let $A'$ be the $(m-1)\times n$ matrix obtained from $A$
by deleting row $m$. Then $A'$ is a maximal 312-avoiding and hence by induction has $2(m-1+n-2)=2(m+n-2)-2$\; 
 1's. Hence $A$ has $2(m+n-2)$\; 1's. Thus we may assume that row $m$ of $A$ contains at least one other 1. Let the first such 1 occur in column $q$.       Then column $q$ hits our zigzag path in a 1. Let the last 1 of the zigzag path in column $q$ be in row $p$. There are two possibilities to consider: (a) column $q$ contains exactly one 1 of the zigzag path  whose row then  contains a crucial 1, or (b) column $q$ contains two or more such 1's and thus column $q$ contains a crucial 1.
 
 First assume (a). Then row $p$ contains a crucial 1, say in column $k$ where we have that $k>q$. Since $A$ is 312-avoiding, the  $(m-1-p)\times (q-1)$ submatrix of $A$ determined by rows $(p+1),\ldots,(m-1)$ and columns $1,\ldots,(q-1)$ is a zero matrix. Let the first 1 of the zigzag path occur in column $l$ where we have that $l<q$;  column $l$ then contains a crucial 1. Then the submatrices $A'$ of $A$ determined by rows $1,\ldots,p,m$ and columns $1,\ldots,q$ is maximal 312-avoiding and so is the submatrix $A''$ of $A$ determined by rows $p,p+1,\ldots,m$ and columns $q,q+1,\ldots,n$. Thus by induction the number of 1's in $A$ equals
 \[2((p+1)+q-2)+2((m-(p-1))+(n-(q-1))-2)-2= 2(m+n-2).\]
 Note that the last $-2$ on the left in this equation accounts for the two 1's that are in the column overlapping $A'$ and $A''$.
 
 Now assume (b). Then similarly   we have maximal $(p+1)\times q$ and $(m-(p-1))\times (n-(q-1))$ submatrices which, by induction, give
 $2(m+n-2)$\ 1's in $A$.
\end{proof}

\begin{corollary}\label{cor:max}
Every maximal $312$-avoiding $m\times n$  $(0,1)$ contains the same number of $1$'s.\hfill{$\Box$}
\end{corollary}

The following algorithm,  carries out in all possible ways, always constructs a 312-avoiding $(0,1)$-matrix with the maximum number of 1's.

\medskip
\centerline{\bf Algorithm for a Maximal 312-avoiding Matrix}
\medskip
\begin{itemize}
\item[\rm (i)] Let $A$ be an  $m\times n$ matrix without any entries specified.
\item[\rm (ii)] Put $(m+n-1)$ 1's in $A$ in the positions of any complete R-L zigzag path connecting the $(1,1)$ then $(1,2)$ positions with the $(m-1,n)$ then $(m,n)$ positions. Also put a 1 in position $(m,1)$. (One could also put 1's in the corner positions but that is not necessary.)
\item[\rm (iii)] Put a 1 in position $(m,q)$ for some $q$ with  $1<q<n$. Let the last 1 in column $q$ of $A$ be in row $p$.
\item[\rm (iv)] Repeat step (iii) recursively to the  submatrix of $A$ determined by rows $1,2,\ldots,p,m$ and columns $1,\ldots,q$, and to the submatrix
 determined by rows $p,p+1,\ldots,m$ and columns $q,q+1,\ldots,n$.
 \item[\rm (v)] Put 0's in all remaining positions of $A$ and  
output the resulting matrix $A$.
\end{itemize}

\begin{example}\label{ex:algorithm} {\rm
Let $m=n=10$ and consider the following complete R-L zigzag path in a $10\times 10$ matrix where we also put a 1 in the $(10,1)$-position and the corner positions:
\[\left[\begin{array}{c|c|c|c|c|c|c|c|c|c}
1&1&1&\mathbf 1&&&&&&\\ \hline
&&&1&&&&&&\\ \hline
&&&1&1&1&\mathbf 1&&&\\ \hline
&&&&&&1&&&\\ \hline
&&&&&&1&&&\\ \hline
&&&&&&1&\mathbf1&&\\ \hline
&&&&&&&1&&\\ \hline
&&&&&&&1&1&\mathbf 1\\ \hline
&&&&&&&&&1\\ \hline
1&&&&&&&&&1\end{array}\right]\rightarrow 
\left[\begin{array}{c|c|c|c|c|c|c|c|c|c}
1&1&1&\mathbf 1&&&&&&\\ \hline
&&&1&&&&&&\\ \hline
&&&1&1&1&\mathbf 1&&&\\ \hline
0&0&0&0&&&1&&&\\ \hline
0&0&0&0&&&1&&&\\ \hline
0&0&0&0&&&1&\mathbf1&&\\ \hline
0&0&0&0&&&&1&&\\ \hline
0&0&0&0&&&&1&1&\mathbf 1\\ \hline
0&0&0&0&&&&&&1\\ \hline
1&0&0&0&1&&&&&1\end{array}\right]\quad (q=5, p=3).
\]
Now repeat with
\[\left[\begin{array}{c|c|c|c|c}
1&1&1&\mathbf 1&\\ \hline
&&&1&\\ \hline
&&&1&\mathbf 1\\ \hline
1&&&&1\end{array}\right]\mbox{ and }
\left[\begin{array}{c|c|c|c|c|c}
1&1&\mathbf 1&&&\\ \hline
&&1&&&\\ \hline
&&1&&&\\ \hline
&&1&\mathbf 1&&\\ \hline
&&&1&&\\ \hline
&&&1&1&\mathbf 1\\ \hline
&&&&&1\\ \hline
1&&&&&1\end{array}\right].\]
Iterating we get as one possibility
\[\left[\begin{array}{c|c|c|c|c}
1&1&1&\mathbf 1&\\ \hline
&1&1&1&\\ \hline
&1&&1&\mathbf 1\\ \hline
1&1&&1&1\end{array}\right]\mbox{ and }
\left[\begin{array}{c|c|c|c|c|c}
1&1&\mathbf 1&&&\\ \hline
1&1&1&&&\\ \hline
1&&1&&&\\ \hline
1&&1&\mathbf 1&&\\ \hline
&&1&1&&\\ \hline
&&1&1&1&\mathbf 1\\ \hline
&&1&1&1&1\\ \hline
1&&1&&&1\end{array}\right]\]
and hence the maximal $10\times 10$ 312-avoiding matrix
\[\left[\begin{array}{c|c|c|c|c|c|c|c|c|c}
1&1&1&\mathbf 1&&&&&&\\ \hline
&1&1&1&&&&&&\\ \hline
&1&&1&1&1&\mathbf 1&&&\\ \hline
&&&&1&1&1&&&\\ \hline
&&&&1&&1&&&\\ \hline
&&&&1&&1&\mathbf1&&\\ \hline
&&&&&&1&1&&\\ \hline
&&&&&&1&1&1&\mathbf 1\\ \hline
&&&&&&1&1&1&1\\ \hline
1&1&&1&1&&1&&&1\end{array}\right].\]

}\hfill{$\Box$}\end{example}

\section{Coda}
In this final section we discuss some related ideas connected with our investigations reported above. We put forth a conjecture and prove it in a special case. 

Let $k\ge 4$ and consider a $k12\cdots(k-1)$-avoiding  $m\times n$ $(0,1)$-matrices $A=[a_{ij}]$. (The case of $k=3$ has been dealt with in Section 3.) If $a_{1n}=1$, then $A$ contains at most $(k-1)(m+n-(k-1))$\ 1's. This follows from the discussion preceding Example \ref{ex:upper} and the fact that a $12\cdots(k-1)$-avoiding $(m-1)\times (n-1)$ $(0,1)$-matrix $A'$ has at most $(k-2)((m-1)+(n-1)-(k-2))$ 1's. The maximality condition implies that all the entries in $A$ in its first row and last column are 1 and hence the number of 1's in $A$ is at most 
\[(m+n-1)+(k-2)((m-1)+(n-1)-(k-2))=(k-1)(m+n-(k-1)).\]
Moreover, there are  $m\times n$ $k12\cdots(k-1)$-avoiding  $m\times n$ $(0,1)$-matrices with this number of 1's.  For example, with $m=n=8$ and $k=5$ we have
\[\left[\begin{array}{c|c|c|c|c|c|c|c}
1&1&1&1&1&1&1&1\\ \hline
&&&1&1&1&1&1\\ \hline
&&1&1&1&1&1&1\\ \hline
&1&1&1&1&1&1&1\\ \hline
1&1&1&1&1&1&&1\\ \hline
1&1&1&1&1&&&1\\ \hline
1&1&1&1&&&&1\\ \hline
1&1&1&&&&&1\end{array}\right] \mbox{with 48\ 1's.}\]
Note that it is only necessary to verify that there is not  a $51234$-sequence of 1's where the first 1 occurs in the $(1,8)$-position.

We conjecture that this works in general for $m$, $n$ and $k\ge 4$.

\begin{conjecture}{\rm \label{con:one}
  A $k12\cdots(k-1)$-avoiding  $m\times n$ $(0,1)$-matrix $A$ contains at most $(k-1)(m+n-(k-1))$ 1's and such a $(0,1)$-matrix with fewer than this number of 1's contains a 0 that can be replaced by a 1 resulting in a $k12\cdots(k-1)$-avoiding matrix.
}\hfill{$\Box$}\end{conjecture}

The upper bound in Conjecture \ref{con:one} can also be established in case that there is a single $0$ in the upper right corner, and we now outline this.  Assume that $A$ has the maximum possible number of 1's. Then the first row and last column are all 1's except for a $0$ in the $(1,n)$-position. For some $p$, the second row of $A$ consists of $(n-p-1)$ \  0's followed by $p$ 1's and then the 1 in the last column. Similarly, for some $q$, the second column has an initial 1 in its first row followed by $q$ 1's and then $(m-1-q)$\ 0's. Thus the 1 in position $(2,n-1)$ is counted by both $p$ and $q$. Without loss of generality we may assume that $p\ge q$. So we have a total of  $(m+n-2)$\  1's  in row 1 and column $n$. In addition in the second row we have $p$ more 1's and in the next-to-last column we have an additional $(q-1)$\  1's.   The number of 1's in the diagonals, beginning with  the first column from its last position to its position in row 2, counted up to and including column $(n-2)$, is at most 
\[1+2+\cdots+(k-3)+(m-(k-2))(k-2).\]
Diagonals corresponding to the entries in row 3 beginning with column 2,  contain at most the following number of 1's:
\begin{itemize}
\item $(n-p+2)$ of them with at most $(k-2)$ 1s;
\item $(p-q)$ of them   with at most $(k-3)$ 1's;
\item $(q-(k-1))$ of them with at most $(k-4)$ 1's;
\item Then diagonals with at most, respectively, $1, 2,\ldots,(k-3)$ 1's.
\end{itemize}
So the number of 1's is at most
\[(m+n-2)+(p+q-1) +2(1+\cdots+(k-3))+(m-k+2)(k-2)+\]
\[(n-p-2)(k-2)+
(p-q)(k-3) +(q-k+1) (k-4).\]
This simplifies to the desired bound:
\[(k-1)(m+n)-(k-1)^2.\]
Notice that  the  $(q-(k+1))$ diagonals with at most $(k-4)$ 1's take into account  the last 1 in column $(n-1)$. Also note that for equality to occur we must have $p,q\ge k-1$.

Another way to extend the classical  notion of pattern-avoidance in permutations to matrices is the following which  for clarity we discuss in the context of 123-avoidance and 312-avoidance but can be formulated more generally.
Let $A$  be  an $n\times n$ $(0,1)$-matrix.
Then   $A$ is {\it $123$-permutation-avoiding} provided every permutation matrix $P$ with $P\le A$ is a $123$-avoiding permutation matrix. Similarly, the matrix $A$ is {\it $312$-permutation-avoiding} provided every permutation matrix $P$ with $P\le A$ is a $312$-avoiding permutation matrix. Thus in both cases, for each $i$, row $i$ of the  matrix $A$ restricts the possibilities for the $i$th term in a permutation of $\{1,2,\ldots,n\}$. If there does not exist   a permutation matrix $P$ with $P\le A$ (entrywise order), then $A$ is both 123-permutation-avoiding and 312-permutation-avoiding. In investigating $n\times n$  $123$-avoiding and $312$-permutation-avoiding $(0,1)$-matrices $A$, it is therefore natural to assume that $A$ has {\it total support} meaning that
$A\ne O$ and   every 1 of $A$ is part of some permutation matrix $P$ with $P\le A$. A property stricter than  having total support is the property of 
 {\it full indecomposability} meaning that  $A\ne O$ has total support and, after row and column permutations,  $A$ is not a non-trivial direct sum  of square matrices (equivalently, deleting any row and any column of $A$, leaves a matrix containing a permutation matrix). But row and column permutations clearly have an effect on pattern-avoiding properties. As a trivial example,
 \[\left[\begin{array}{cc} 0&1\\ 1&0\end{array}\right] \] is 12-permutation-avoiding and  12-avoiding, but permuting to get
 \[\left[\begin{array}{cc} 1&0\\ 0&1\end{array}\right]\]
 results in a matrix which is not 12-permutation-avoiding matrix and not 12-avoiding.
 
 \begin{question} {\rm \label{qu:count}
 What is the maximum number of 1's in an $n\times n$ $(0,1)$-matrix  which is 123-permutation-avoiding and has total support  (resp.,  $312$-permutation-avoiding)?
 }\hfill{$\Box$}\end{question}

Recall that the {\it permanent} of a square $(0,1)$-matrix $A$ is given by 
\[\mbox{per}(A) =\sum_{\pi \in {\mathcal S}_n} \prod_{i=1}^n a_{i,\pi (i)}\]
and counts the number of permutation matrices $P\le A$.

\begin{question}{\rm \label{qu:perm} What is the maximum permanent of an $n\times n$ $(0,1)$-matrix   which is 123-avoiding (resp.,  123-permutation-avoiding)? 
 These questions are equivalent to finding how many   123-avoiding  permutations can be in $n\times n$  $123$-avoiding (respectively,
$123$-permutation-avoiding ) $(0,1)$-matrices. There are similar questions with  $312$ replacing $123$. Indeed, one may ask for the number of 123-avoiding and 312-avoiding permutations in any $n\times n$ (0,1)-matrix, indeed the number of $\sigma$-avoiding and $\sigma$-avoiding-permutations in any $n\times n$ (0,1)-matrix where $\sigma\in {\mathcal S}_k$.

This suggests the following definition.
Let $k\le n$ and  $\sigma\in {\mathcal S}_k$, and 
let ${\mathcal S}_n({\bar \sigma)}$ be the set of $\sigma$-avoiding permutations in ${\mathcal S}_n$. Define the {\it $\sigma$-avoiding permanent} of an $n\times n$ matrix $A$ to be
\[\mbox{per}_{\bar\sigma}(A) =\sum_{\pi \in {\mathcal S}_n(\bar \sigma)} \prod_{i=1}^n a_{i,\pi (i)}.\]
Taking $A$ to be $J_n$  the $n\times n$ matrix of all 1's
and $\sigma=(1,2,3) \mbox{ or } (3,1,2)$, we have 
\[\mbox{per}_{\bar\sigma}(J_n) =C_n=
\frac{{{2n}\choose n}}{n+1}.\]
If $A$ is $\sigma$-permutation-avoiding, then $\mbox{per}(A)=
\mbox{per}_{\sigma}(A)$. Are there other families of $n\times n$ matrices $A_n$  for which $\mbox{per}_{\bar\sigma}(A_n)$ has a nice formula?}\hfill{$\Box$}\end{question}

\begin{question}{\rm \label{qu:maxper}
Let $\sigma\in {\mathcal S}_k$. What is the maximum $\sigma$-avoiding permanent of an $n\times n$ fully indecomposable  $(0,1)$-matrix $A$ and which matrices attain this maximum?
What is the maximum number of 1's in an $n\times n$ fully indecomposable  $\sigma$-permutation-avoiding $(0,1)$-matrix $A$ and which matrices attain this maximum?}\hfill{$\Box$}
\end{question}

if $\sigma=312$, a
 guess is that the maximum $\sigma$-avoiding permanent occurs for the following matrix (illustrated for $n=5$):
\[\left[\begin{array}{ccccc}
1&1&1&0&0\\ 
1&1&1&1&0\\ 
1&0&1&1&1\\ 
1&0&0&1&1\\ 
1&0&0&0&1\end{array}\right],\]
which has permanent equal to 12.

Another related question is the following.

\begin{question}\label{qu:two} {\rm Let $k$ and $n$ be positive integers with $k<n$.
Let $\sigma\in {\mathcal S}_k$.
  When can a $\sigma$-avoiding permutation of a subset of the integers in $\{1,2,\ldots,n\}$ be extended to a $\sigma$-avoiding permutation of $\{1,2,\ldots,n\}$? For example, if $k=3$ and $\sigma=(3,1,2)$, then $(4,6,1)$ can be extended to $(2,{\bf 4},5,{\bf 6},3,{\bf 1})$ which is $312$-avoiding.
}\hfill{$\Box$}
\end{question}

More background material on $(0,1)$-matrices can be found in \cite{RAB}.
\smallskip

\noindent
{\bf Acknowledgement:} We are very  grateful to Balazs Keszegh for pointing out to us the  somewhat larger context for our investigations, including  the references  \cite{Furedi, BK, Marcus,SP1, SP2}, and for other comments.


\begin{thebibliography}{99}
\bibitem{Bona} M.~B\'ona, {\it Combinatorics of Permutations}, 2nd edition, CRC Press 2012, Chapter 4.
\bibitem{RAB} R.A.~Brualdi, {\it Combinatorial Matrix Classes}, Cambridge Univ. Press, 2006.
\bibitem{Furedi} Z.~F\H uredi and P.~Hajnal, Davenport-Schinzel theory of matrices, {\it Discrete Math.}, 103 (1992), 233-251.
\bibitem{BK} B.~Keszegh, Forbidden submarices in 0-1 matrices, Thesis, E\H otv\H os Lor\'and University, 2005.
\bibitem{Marcus} A.~Marcus and G.~Tardos, Excluded pemutation matrices and the Stanley-Wilf conjecture, {\it J. Combin. Theory Ser. A}, 107 (2004), 153--160.
\bibitem{SP1} S.~Pettie, 
Degrees of nonlinearity in forbidden 0-1 matrix problems, {\it Discrete Math.} 311 (2011), no. 21, 2396–2410.
\bibitem{SP2} S.~Pettie,
Generalized Davenport-Schinzel sequences and their 0-1 matrix counterparts,
{\it J. Combin. Theory Ser. A},  118 (2011), no. 6, 1863–1895.
\bibitem{GT}
G.~Tardos, 
On 0-1 matrices and small excluded submatrices, {\it  J. Combin. Theory Ser. A}, 111 (2005), no. 2, 266–288.

\end{thebibliography}
\end{document}